\documentclass[12pt]{article}          
\usepackage{amsmath,amsthm,amssymb}
\usepackage{pifont}  
\usepackage{hyperref}
\usepackage{array,color,colortbl,mathrsfs,longtable} 
\usepackage[noend]{algorithmic}
\usepackage{algorithm}    
\usepackage{fullpage}
\renewcommand{\leq}{\leqslant}

\renewcommand{\le}{\leqslant}

\theoremstyle{plain}
\newtheorem{theorem}{Theorem}[section]
\newtheorem{lemma}[theorem]{Lemma}

\theoremstyle{definition}

\renewcommand{\leq}{\leqslant}  

\providecommand\Z{\mathbb{Z}} 
\providecommand\OO{\mathscr{O}}
\providecommand\M{\mathscr{M}}
\providecommand\sym{\mathcal{S}}

\providecommand\spc{\mathscr{T}}

\newcommand{\atp}{\mathop{\mathrm{atp}}}
\newcommand{\ptp}{\mathop{\mathrm{par}}}

\DeclareMathOperator{\AS}{AS}
\DeclareMathOperator{\AL}{AL}
\DeclareMathOperator{\RS}{RS}
\DeclareMathOperator{\RL}{RL}

\providecommand\MOLS[2]{\ensuremath{#1\textrm{-}\mathrm{MOLS}(#2)}}
\providecommand\maxMOLS[2]{\ensuremath{#1\textrm{-}\mathrm{maxMOLS}(#2)}}

\providecommand{\shdA}{\cellcolor[gray]{.8}}
\providecommand{\shdB}{\cellcolor[gray]{.5}}
\providecommand{\shdC}{\cellcolor[gray]{.7}}

\providecommand{\and}{\ensuremath{\wedge}}

\def\tref#1{Theorem~$\ref{#1}$}
\def\Tref#1{Table~$\ref{#1}$}
\def\lref#1{Lemma~$\ref{#1}$}

\def\eref#1{$(\ref{#1})$}
\def\sref#1{Section~$\ref{#1}$}

\def\spacer{{\vrule height 2.75ex width 0ex depth1.5ex}}
\def\tspacer{{\vrule height 2.25ex width 0ex depth0ex}}

\begin{document}
\title{Enumeration of MOLS of small order\footnote{2010 
AMS Subject Classification 05B15 (62K99).}
\footnote{Keywords: latin square; MOLS; transversal; plex; orthogonal mate 
}
}

\author{Judith Egan and Ian M. Wanless\\
\small School of Mathematical Sciences\\[-0.8ex]
\small Monash University\\[-0.8ex]
\small VIC 3800 Australia\\
\small \texttt{\{judith.egan, ian.wanless\}@monash.edu}
}

\date{}
\maketitle
\begin{abstract}
  We report the results of a computer investigation of sets of
  mutually orthogonal latin squares (MOLS) of small order.  For $n\le9$ we 
\begin{enumerate}\addtolength{\itemsep}{-0.25\baselineskip}
\item Determine the number of orthogonal mates for
  each species of latin square of order $n$.
\item Calculate the proportion of latin squares of order $n$ that have
  an orthogonal mate, and the expected number of mates when a square is
  chosen uniformly at random.
\item Classify all sets of MOLS of order $n$ up to various different
  notions of equivalence.
\end{enumerate}

We also provide a triple of latin squares of order 10 that is the closest
to being a set of MOLS so far found.
\end{abstract}



\section{Introduction}\label{s:intro}

A latin square of order $n$ is an $n \times n$ matrix in which $n$
distinct symbols are arranged so that each symbol occurs once in each
row and column.  Two latin squares $A = [a_{ij}]$ and $B = [b_{ij}]$
of order $n$ are said to be \emph{orthogonal} if the $n^2$ ordered
pairs $(a_{ij},b_{ij})$ are distinct.  A set of MOLS (\emph{mutually
orthogonal latin squares}) is a set of latin squares in which each
pair of latin squares is orthogonal. The primary aim of this paper is
a thorough computational study of all sets of MOLS composed of 
latin squares of order at most $9$.

We use \MOLS{k}{n} as shorthand for $k$ MOLS of order $n$.
If $A$ and $B$ are orthogonal then $B$ is an \emph{orthogonal mate} of
$A$, and vice versa.  A latin square with no orthogonal mate is called
a \emph{bachelor} square \cite{vanRees1990}.  A set of
\maxMOLS{k}{n} is a set of \MOLS{k}{n} that is {\em maximal} in the
sense that it is not
contained in any set of \MOLS{(k+1)}{n}. The existence problem
for \maxMOLS{1}{n} (i.e.~bachelor latin squares) was solved in
\cite{Evans,WW}. For the most recent progress on the existence of
\maxMOLS{2}{n}, see \cite{DanzigerWanlessWebb2011}.
For $k>2$ our knowledge is quite patchy; see \cite{cd07} for a summary.
However, for $1\le k<n \leq 9$, the question of whether or not
there exists a set of \maxMOLS{k}{n} is completely answered due to the
collective works of Drake \cite{Drake1977}, Jungnickel and Grams
\cite{JungnickelGrams1986}, and Drake and Myrvold
\cite{DrakeMyrvold2004}. 
In \cite[p190]{cd07} there is
a table of values of $k$ for which \maxMOLS{k}{n}
are known to exist for $n \leq 61$, but missing from it is the case
$(k,n)=(4,9)$ due to \cite{DrakeMyrvold2004}.

A \emph{transversal} in a latin square of order $n$ is a selection of
$n$ distinct entries in which each row, column and symbol has exactly
one representative.  A partition of a latin square of order $n$ into
$n$ disjoint transversals is called a \emph{$1$-partition}.  A latin
square has an orthogonal mate if and only if it possesses a
$1$-partition \cite[p155]{dk74}. More generally, a $p$-plex is a
selection of $pn$ distinct entries in which each row, column and
symbol has exactly $p$ representatives~\cite{w02}.  A partition of a
latin square into disjoint $p$-plexes is called a
\emph{$p$-partition}.  We discuss algorithms for finding $p$-plexes
and $p$-partitions in \sref{s:algo}.

For a latin square $L$ of order $n$, we define $\theta=\theta(L)$ to be the
number of $1$-partitions of $L$. Put another way, $\theta(L)$ is the
number of orthogonal mates of $L$ that have their first row equal to
$[1,2,\dots,n]$.  The number of transversals in a latin square is well
known to be a species invariant and it follows from the same reasoning
that $\theta(L)$ is also a species invariant. The definition
of species (also known as ``main class'') will be given in 
\sref{s:equiv}.

The results in \cite{ew2011lswrt} include a classification of the
species of orders up to $9$ by whether or not they possess an
orthogonal mate (see \cite{w02} for an earlier table, giving similar
data for orders up to $8$).  In \sref{s:nummates} we report data on
the number of mates for all species of orders up to $9$.  We then
calculate the expected value of $\theta(L)$ for $L$ selected uniformly
at random from the latin squares of order $9$.  For orders $n \leq 8$
this information is available in \cite{ew2011indiv}.
 
A set of \MOLS{(n-1)}{n} is also known as a \emph{complete} set of
MOLS and is equivalent to an affine plane of order~$n$
\cite{Bose1938}.  In 1896, Moore showed that the maximum cardinality
of a set of MOLS of order~$n$ is $n-1$, and that this upper bound is
achieved if $n$ is a prime power \cite{Moore1896}.  The converse,
whether \MOLS{(n-1)}{n} exist only if $n$ is a prime power, is a
prominent open problem.  Further information and partial results can
be found in \cite{cd07,dk74,lm98}.  We use the name \emph{planar latin
  square} for any latin square that is a member of some set of
\MOLS{(n-1)}{n}.  We will refer to the species of planar latin squares
of order 9 by the labels $a,b,\dots,k$ given to them by Owens and
Preece \cite{OP95}. In \sref{s:MaxMOLS} we will investigate
the role that these squares play in forming sets of maxMOLS that are
not complete.

For excellent general references on enumeration problems
of the type we undertake, see \cite{KO06,Ost05}.
For recent related work on enumerating mutually orthogonal latin cubes,
see \cite{KO14}.

The outline of the paper is as follows. In \sref{s:equiv} we define
our basic terminology and establish the different notions of equivalence
that we want to use when counting MOLS. In \sref{s:algo} we describe
the basic algorithms that we used for counting MOLS, as well as providing
the mathematical theory that underpins those algorithms. The case
of \MOLS{2}{9} is the most difficult that we treat and it requires some
special considerations that are described in \sref{s:pairs9}.
In the process we give our first data from the computations, which
is a classification of the \MOLS{2}{9} according to how many
symmetries they possess. In \sref{s:nummates} we provide data on how many
orthogonal mates each latin square of order up to 9 possesses and
identify the squares with the most mates. We also calculate the
probability that a random latin square will have a mate and the 
expected number of mates. The main data is provided in \sref{s:MaxMOLS},
where we provide counts of MOLS and maxMOLS classified according to
the many notions of equivalence defined in \sref{s:equiv}. We also
provide information on many other matters such as the number of disjoint
common transversals, which species of latin squares are most prevalent
in the sets of MOLS, how many MOLS contain planar latin squares and so on.
In \sref{s:safety} we describe a number of ways in which we have crosschecked
our data in order to reduce the chances of errors. Finally, in \sref{s:ten}
we give three latin squares of order 10 that are closer to being
a set of \MOLS{3}{10} than any previously published.

There are numerous tables in this paper which report counts of
different types of MOLS. In every table we use the convention that a
blank entry should be interpreted as zero, meaning there are
definitely no MOLS in that category.
    
\section{Symmetries and notions of equivalence}\label{s:equiv}

The number of latin squares of order $n$ grows rapidly as $n$ grows
and is only known \cite{MW05,HKO11} for $n\le11$. Little is known
about the number of sets of MOLS, although it is clear that it too
increases very quickly \cite{DG13}.  To cope with this ``combinatorial
explosion'' it is vital to use a notion of equivalence to classify the
different possibilities.  Several different notions of equivalence are
outlined in this section.  We used the weakest notion (that is, the
one that considers the most things equivalent) in the first instance
to compile a list of representatives from equivalence classes.  From
these we can then infer the number of equivalence classes using
stronger notions of equivalence. With this strategy, the
computational limit is the MOLS of order~$9$.

Taking care in our enumeration, we will sometimes need to distinguish
between sets of MOLS and lists of MOLS (a list is an ordered set). The
distinction will become important shortly. The definitions below are
intended to apply to any number of MOLS, including the (arguably
degenerate) case of MOLS that consist of a single latin square.
In that case, of course, sets and lists are the same thing.

To introduce our various notions of equivalence, it is useful to
discuss a well-known relationship between MOLS and orthogonal
arrays.  Let $S$ be a set of cardinality $s$ and let $O$ be an $s^2\times
k$ array of symbols chosen from $S$.  If, for any pair of columns of
$O$, the ordered pairs in $S \times S$ each occur exactly once among
the rows in those chosen columns, then $O$ is an \emph{orthogonal array}
of strength 2 and index 1. We will omit further reference
to the strength and index, since we will not need orthogonal arrays
with other values of these parameters.  See \cite{cd07} for further
details and background on orthogonal arrays.

A list $L_1,L_2,\ldots,L_k$ of MOLS of order $n$ can be used to
build an $n^2\times(k+2)$ orthogonal array as follows.  For each row
$r$ and column $c$ of the latin squares there is one row of the 
orthogonal array equal to
\begin{align*}
\left(
		r,\ 
		c,\ 
		L_1[r,c],\ 
		L_2[r,c],\ 
		\ldots,\ 
		L_k[r,c]
\right),
\end{align*}
where $L_i[r,c]$ is the symbol in row $r$, column $c$ of the square
$L_i$.  Moreover, the process is reversible, so that any
$n^2\times(k+2)$ orthogonal array can be interpreted as a list of $k$
MOLS of order $n$.  In other words, orthogonal arrays 
correspond to lists of MOLS (the correspondence is not one-to-one,
but only because permuting the rows of the orthogonal array changes
the array but does not affect the MOLS).  We will sometimes talk of an
orthogonal array representing a set of MOLS. In such cases we will
mean that any order can be imposed on the set to make it a list,
and it makes no material difference which order is chosen.

We call two orthogonal arrays of the same dimensions {\em equivalent}
if they are the same up to permutation of the rows and columns of the
array and permutations of the symbols within each column of the array.
We define two lists of MOLS to be \emph{paratopic} if they define
equivalent orthogonal arrays in this sense.

Let $\sym_n$ denote the symmetric group of degree $n$.
Viewed another way, paratopism is an action of the wreath product
$\sym_n\wr\sym_{k+2}$ on lists of \MOLS{k}{n}, where each copy of
$\sym_n$ permutes the symbols in one of the columns of the
corresponding orthogonal array, while $\sym_{k+2}$ permutes the
columns themselves.  An orbit under paratopism is known as a {\em
  species} of MOLS.  The stabiliser of a list of MOLS $M$ under
paratopism will be called its {\em autoparatopism group}, which we
denote by $\ptp(M)$.  We say that a group is {\em trivial} if it has
order 1, and {\em non-trivial} otherwise.  Lists of MOLS that have
trivial autoparatopism group are {\em rigid}, all other MOLS will be
called {\em symmetric}.

We call two lists of MOLS {\em isotopic} if they define the same
orthogonal array, up to permutation of the symbols within each column
of the array and permutation of the rows of the array. In latin
squares terminology, we are allowing uniform permutation of rows and
columns of the squares as well as permutation of the symbols within
each square.  We call two lists of MOLS {\em trisotopic} if they are
isotopic, or if swapping the first two columns of the orthogonal array
for one of the lists makes it isotopic to the other. In latin squares
terminology, trisotopism is the same as isotopism except that we also
allow the squares to be transposed in the usual matrix sense.

Isotopism can be viewed as an action of the direct product of $k+2$
copies of $\sym_n$ on lists of \MOLS{k}{n}.  The stabiliser of a list
$M$ of MOLS under isotopism is known as the {\em autotopism group} of
$M$, which is denoted $\atp(M)$.  The orbit of $M$ under isotopism is
known as its isotopism class -- it is the set of all lists of MOLS that
are isotopic to $M$.  Similarly, the trisotopism class of $M$ is the set
of all lists of MOLS that are trisotopic to $M$.

We call two sets of MOLS isotopic (respectively, trisotopic,
paratopic) if there is any way in which they can be ordered so that
the resulting lists of MOLS are isotopic (respectively, trisotopic,
paratopic). Again, we define the isotopism class (respectively,
trisotopism class, species) of a set $M$ of MOLS to be the set of all
sets of MOLS that are isotopic (respectively, trisotopic, paratopic)
to $M$.  We will frequently discuss species of MOLS without specifying
whether the MOLS are lists or sets. This is appropriate since species of
lists of MOLS correspond one-to-one to species of sets of MOLS, simply
by ``forgetting'' the order of the lists.  Similar statements fail for
isotopism classes and trisotopism classes -- there are typically more of
these for lists of MOLS than for the corresponding sets of MOLS.

For complete sets/lists of MOLS there are also geometric notions of
equivalence.  We define two complete sets/lists of MOLS to be {\em
  PP-equivalent} if they correspond to isomorphic projective planes,
and to be {\em DPP-equivalent} if the projective planes they define
are, up to isomorphism, either equal or dual.

The strongest possible notion of equivalence for MOLS is equality,
when considered as lists or sets. A list of MOLS is
{\em reduced} if all squares in the set have their first row in order
and the first square has its first column in order. A set of MOLS is
{\em reduced} if an ordering can be put on it to make it a reduced
list of MOLS.

It should be clear from the definitions that equality is a refinement of
isotopism equivalence which in turn is a refinement of trisotopism
equivalence which in turn is a refinement of paratopism equivalence. For
complete sets of MOLS, paratopism equivalence is a refinement of
PP-equivalence which is a refinement of DPP-equivalence.  The
relationship between projective planes, MOLS and different notions of
equivalence was studied by Owens and Preece
\cite{Owens1992,OP95,OP97}.  Our enumerations
confirm and extend a number of their results.

Some of our terminology follows the pioneers of the subject, such as
Norton \cite{Nor39}, who used ``species'' in our sense for single
latin squares and also for larger sets of MOLS. Another term that we
want to borrow from \cite{Nor39} is the notion of an {\em aspect}.  An
{\em aspect} of a list or set of MOLS is obtained by selecting 3
columns of the corresponding orthogonal array, then interpreting the
result as a latin square. In our work we will only care about which
species each aspect is in, so we will talk of there being
${k+2\choose3}$ aspects for a set or list of $k$ MOLS. In other words,
aspects will be considered to be the same if they use the same 
3 columns of the orthogonal array, but in a different order.

The orthogonal array interpretation of a set of MOLS provides an easy
mechanism for converting any set of $k$-MOLS containing a particular
latin square $L$ into another set of $k$-MOLS that contains $L'$,
where $L'$ is any latin square in the same species as $L$. What is not
so obvious is that the conversion may change the species of some or
all of the $k-1$ latin squares in the sets of MOLS other than
$L$. Variation of the species of latin squares among paratopic sets of
MOLS was observed by Owens and Preece \cite{OP95,OP97} in their study
of complete sets of MOLS of order $9$ obtained from affine planes of
order $9$.  See also \cite{MaenhautWanless2004} for an explicit
example.  However, with that caveat, to enumerate species of MOLS it
is sufficient to start with a set of representatives of species of
latin squares and find the sets of MOLS that they are contained in.
The details of how we did this will be discussed in \sref{s:algo}.

Suppose that $M$ is a set of \MOLS{k}{n} and $O$ is the
corresponding orthogonal array.  A {\em common transversal} for $M$ is
a selection of $n$ of the rows of $O$ in which no two rows share
the same symbol in any column. In other words, in the $n\times(k+2)$
subarray of $O$ formed by the chosen rows, each column is a
permutation of the $n$ symbols in $O$. A particularly important
consideration is whether $O$ can be partitioned into subarrays of this
type.  The set $M$ of \MOLS{k}{n} has a set of $n$ disjoint common
transversals if and only if $M$ is a subset of some set of \MOLS{(k+1)}{n},
in other words, it is not maximal.



We finish the section with an example that illustrates why we need to
carefully distinguish between sets and lists when enumerating MOLS.
All calculations in this example will be in $\Z_5$.
Define $L_x$ to be the latin square of order $5$ whose entry in
cell $(i,j)$ is $xi+j$. It is easy to see that $L_1$, $L_2$, $L_3$
and $L_4$ are mutually orthogonal. However, the list $(L_1,L_4)$ is
isotopic to $(L_4,L_1)$, by applying the permutation $x\mapsto 4x$
to the rows of both squares. In contrast, $(L_1,L_2)$ is not isotopic,
as a list, to $(L_2,L_1)$, even though the corresponding sets are
clearly equal. Hence, the sets of MOLS $\{L_1,L_2\}$ and $\{L_1,L_4\}$
correspond to a total of three different lists, up to isotopism.
All three lists are in the same species.

This example illustrates an interesting point regarding autotopism
groups. We have been careful to define $\atp$ (and $\ptp$) only for
lists of MOLS, where the group actions that we have described are
well-defined. It is tempting to define autotopisms of sets of MOLS by
considering the autotopisms of a corresponding list of MOLS.  If we do
this in the above example, the set $\{L_1,L_4\}$ seems to have twice
as many autotopisms as $\{L_1,L_2\}$, since there are the autotopisms
that preserve the list $(L_1,L_4)$, as well as those that map $L_1$ to
$L_4$ and vice versa. This would mean that the number of autotopisms
is not a species invariant for sets of MOLS. In any case, we do not
need a notion of an autotopism group for sets of MOLS in this work.

\section{Basic Algorithms}\label{s:algo}

In this section we discuss the algorithms that we used for enumerating
the MOLS of a given order and testing them for equivalence.  The first
task was to obtain a set of species representatives for the MOLS. Next
we used these species representatives to count the isotopism classes and
trisotopism classes for sets and lists of MOLS. Lastly, we calculated
the number of reduced MOLS using two theorems that we prove 
at the end of this section. The techniques described in this section
were feasible in most cases. The case $(k,n)=(2,9)$ required
some additional considerations, which are described in \sref{s:pairs9}.

We began with a set of species representatives for latin squares of
order $n$. For $n\le8$ these are available in many places, including
\cite{BDMLS}. For order 9 there are too many to store, so 
we generated the species representatives 
``on the fly'', using a program written for \cite{MMM07}. 
Our first task reduces to the problem of finding a set of species
representatives for sets of \MOLS{(k+1)}{n} given a set of species
representatives for sets of \MOLS{k}{n}. This requires us to find all
possible $1$-partitions of each \MOLS{k}{n} in turn.  Except when
$(n,k+1)=(9,2)$, the resulting number of \MOLS{(k+1)}{n} was small
enough to screen for isomorphism, in a way that we describe below, 
to select the required set of species representatives.



In \cite{ew2011indiv} we conducted an exhaustive study of the
indivisible partitions of latin squares of orders~$\le8$. The algorithm
used for finding partitions in that study included the special case
of 1-partitions. Since it is almost as simple to describe how to find
$p$-partitions for a general $p$, we describe this more general
algorithm now.


The first step was to generate and store all of the $p$-plexes.  This
was possible for the cases encountered in \cite{ew2011indiv} and in
the present work, but for most larger squares the number of $p$-plexes
would be too large to store. To generate all $p$-plexes we used a
simple backtracking algorithm, aided by bit-arithmetic. Given a list
$L_1,\dots,L_k$ of \MOLS{k}{n} we first computed an $n\times n$ array
$U$ of bitstrings.  The entry in cell $(r,c)$ of $U$ was
$2^c+\sum_{i=1}^k2^{in+L_i[r,c]}$, using $\{0,\dots,n-1\}$ to index
rows, columns and symbols. The backtracking worked row by row,
adding to our plex all allowable choices of $p$ cells from a row.  To
keep track of what is allowable, we maintained one bitstring for each
$i=1,2,\dots,p$ which recorded which symbols from each latin square
and/or columns were already represented $i$ times in our plex. These
bitstrings were updated using the matrix $U$.  Each plex that was
found was stored as a bitstring of $n^2$ bits saying which cells were
included in the plex. This allowed rapid pairwise comparison to see if
two plexes were disjoint, or similarly, to check if one plex was a
subset of another. The latter question was vital when testing
divisibility of plexes in \cite{ew2011indiv}, but is not so important
to us here.

In the process of generating the plexes, we also computed a look-up
table $T$ which recorded, for each plex $P$ and row $r$, the index of
the plex $T[P,r]$ which was the first plex in the catalogue after $P$
whose cells in row $r$ were different (in at least one place) 
from the ones used in $P$.  This
look-up table greatly sped up the second stage, which was the finding
of all $p$-partitions. Here again we used backtracking.  We built each
$p$-partition one $p$-plex at a time. However, if we found that a
particular plex $P$ could not be added to our partition, then we
located the first row $r$ in which $P$ intersected with the plexes already
chosen, then skipped forward in the
catalogue to $T[P,r]$, the next plex that might have a chance of being
compatible. As an example, consider the process of choosing
transversals to make a $1$-partition. We end up choosing the
transversals in order of which cell they use in the first
row. Skipping forward using $T$ is one way to ensure that we do
not waste time considering transversals that clash in the first row
with a transversal that we have already chosen. Note that if we are
only interested in finding $1$-partitions then we may enforce that the
$i^{\rm th}$ transversal that we choose uses the $i^{\rm th}$ cell in
the first row. However, if we are looking for the largest number of
disjoint transversals then we may only assume that each transversal
that we choose uses a cell in the first row to the right of that used
by the previous transversal.

With the above algorithm we were able to find all 1-partitions of a
set of MOLS. In particular, of course, if there are no 1-partitions
then the MOLS are maximal. It is worth making some comments on an
alternative approach to finding 1-partitions. Finding all the
transversals in a latin square can be viewed as an instance of the
\texttt{exact cover} problem \cite{KO06,Ost05}.  Once the transversals
have been generated and stored, finding all the 1-partitions is
another instance of \texttt{exact cover}.  A solver for \texttt{exact
  cover}, called \texttt{libexact}, is available at \cite{KP14}. It
uses what \cite[p.149]{KO06} describes as ``an algorithm that lacks
serious competitors''.  However, as often happens, we were able to
beat the general purpose algorithm by exploiting the particularities
of our setting. We found that our algorithm was
faster than \texttt{libexact} by a factor of 2 for average latin
squares of order 9, and faster by a factor of 7 for the latin squares
with the most orthogonal mates (the group tables).  The time taken to
find the transversals was negligible compared to the time taken to
find the 1-partitions. On a standard desktop PC, our code took roughly
17 seconds to find the $12445836$ mates for the elementary abelian
group of order 9, but could process over 1200 typical latin squares of
order 9 per second.

\bigskip

Next we discuss the issue of equivalence testing for MOLS. For this
task we used \texttt{nauty} \cite{nauty} to canonically label our MOLS, which
could then be compared pairwise to see if their canonical forms were
equal. This is a standard way to employ \texttt{nauty}, but we needed
to encode our MOLS as a graph so that \texttt{nauty} could be applied. This
is easiest to describe by considering the orthogonal array
representation for the MOLS.

Suppose that we have an $n^2\times k$ orthogonal array $O$
corresponding to a list of MOLS $M$. We now define an undirected graph
$G_O$ corresponding to $O$. The vertices of $G_O$ are of three
types. There are:
\begin{itemize}\addtolength{\itemsep}{-0.25\baselineskip}
\item
$k$ type 1 vertices that correspond to the columns of $O$,
\item
$kn$ type 2 vertices that correspond to the symbols in each of
the columns of $O$, and
\item
$n^2$ type 3 vertices that correspond to the rows of $O$.
\end{itemize}
Each type 1 vertex is joined to the $n$ type 2 vertices that
correspond to the symbols in its column. Each type 3 vertex is
connected to the $k$ type 2 vertices that correspond to the symbols
in its row. There are no more edges. 
Vertices are coloured according to their type so that isomorphisms
are not allowed to change the type of a vertex. It is now routine
to check the following key facts (that generalise 
observations from \cite{MMM07}, which dealt with the case $k=3$):
\begin{itemize}\addtolength{\itemsep}{-0.25\baselineskip}
\item  the automorphism group of $G_O$ is isomorphic to $\ptp(M)$.

\item If $G_{O'}$ is the graph corresponding to another orthogonal
  array $O'$ then $G_O$ is isomorphic to $G_{O'}$ if and only if $O$
  is paratopic to $O'$.
\end{itemize}
This shows how we tested paratopism (of sets or lists) of MOLS.
Moreover, we can test the other equivalence relations we need 
by altering the colouring of the type 1 vertices.  Suppose that the
first two type 1 vertices correspond to the rows and columns of the
latin squares respectively. Then to test isotopism of lists of MOLS we
give each type 1 vertex a different colour.  To test isotopism of sets
of MOLS we give the first two type 1 vertices different colours, then
all remaining type 1 vertices are given a third colour.  In both cases,
trisotopism is the same as isotopism except that the first two type 1
vertices get the same colour. Since \texttt{nauty} looks only for colour
preserving isomorphisms, this allowed us to test the different notions
of equivalence that we needed.  We could simply take each species
representative, reorder the columns of their orthogonal array in all
ways that might plausibly be inequivalent, then test with \texttt{nauty} which
ones were in fact inequivalent.

One other point bears mentioning, which is that \texttt{nauty} can be
dramatically quickened by use of vertex invariants \cite{nauty}.  We
trialed several invariants of which the fastest was \texttt{cellfano2},
which is one of the invariants that ships in the current distribution
of \texttt{nauty}.

\medskip

By using \texttt{nauty} as described above, we were able to compile
catalogues of representatives for species, trisotopism classes and
isotopism classes of MOLS for all cases except when $(k,n)=(2,9)$. We
did not do any computations of PP or DPP equivalence, since
classification of sets of MOLS under those notions is well known
\cite{cd07} for orders up to and including $9$.  So it only remains to
discuss how we counted reduced MOLS. For this we employed the
following theorems.  In the next result, a {\em class} of MOLS should be
interpreted as containing both lists and sets of MOLS, with a list
being a member of the class if and only if the corresponding set is in
the class.

\begin{theorem}\label{t:switch}
  Suppose $1\le k<n$. Let $\M$ be any class of \MOLS{k}{n} that is
  closed under isotopisms.
  Let $\RS_{\M}$ and $\RL_{\M}$ be the number of reduced sets and
  reduced lists, respectively, in $\M$. Let $\AS_{\M}$ and
  $\AL_{\M}$ be the corresponding numbers of arbitrary $($that is, not
  necessarily reduced\/$)$ sets and lists. These numbers are related by
\[
(k-1)!n!^k(n-1)!\RS_{\M}
=n!^k(n-1)!\RL_{\M}
=\AL_{\M}=k!\AS_{\M}.
\]
\end{theorem}

\begin{proof}
Since orthogonal latin squares cannot be equal, the last equality is
immediate. The first equality is similar, given that any reduced set
or list of MOLS contains a unique reduced latin square. 

To prove the middle inequality, we construct a bipartite multigraph
where the two vertex parts are, respectively, the reduced and arbitrary 
lists of \MOLS{k}{n}. Let $R$ be any reduced list of \MOLS{k}{n}. 
We add one edge from $R$ for every isotopism, with the other end of
the edge being the list that results from applying the isotopism
to $R$. Hence the degree of $R$ will be $n!^{k+2}$, the number of
possible isotopisms.

Now consider $A$, an arbitrary list of \MOLS{k}{n}. The degree of $A$
will be the number of isotopisms that can be applied to $A$ to produce
a reduced list of MOLS.  Such an isotopism is determined by the
permutation it applies to the columns of the squares in $A$, and which
row becomes the first row. Once these choices are made, there is a
unique way to permute the symbols in each square to get the first row
in order and a unique way to permute the remaining rows to get the
first column of the first square in order.  Hence there are $n!\,n$
possible choices, and each produces exactly one reduced list. In other
words, the degree of $A$ is $n!\,n$.  Thus our multigraph is
bi-regular, so the sizes of the two vertex parts are in the opposite
ratio to the degrees of the vertices in those parts, yielding the
claimed equality.
\end{proof}

\tref{t:switch} deals with classes of \MOLS{k}{n} that are closed under
isotopisms. An important example is the class of all \MOLS{k}{n}.
In that case we will write $\RS_{k,n}$, $\RL_{k,n}$, $\AL_{k,n}$, $\AS_{k,n}$
instead of $\RS_{\M}$, $\RL_{\M}$, $\AL_{\M}$, $\AS_{\M}$, respectively.


\begin{theorem}\label{t:RSfromreps}
  Suppose $1\le k<n$. 
  Let $\OO$ be a set of sets of \MOLS{k}{n} such that no two elements of
  $\OO$ are paratopic. The number of reduced sets of \MOLS{k}{n}
  that are paratopic to some member of $\OO$ is
  \[
  n!\,n(k+2)(k+1)k\sum_{M\in\OO}\frac{1}{|\ptp(M)|}.
  \]
\end{theorem}

\begin{proof}
It suffices to prove the case when $\OO$ contains a single set $M$ of
$\MOLS{k}{n}$. Let $A$ be an orthogonal array representation of $M$.
By the Orbit-Stabiliser Theorem,
the number of orthogonal arrays equivalent to $A$ is 
\begin{equation}\label{e:ptp}
\frac{|\sym_n\wr\sym_{k+2}|}{|\ptp(M)|}=\frac{n!^{k+2}(k+2)!}{|\ptp(M)|}.
\end{equation}
The number of reduced sets of \MOLS{k}{n} paratopic to $M$ is
obtained by dividing \eref{e:ptp} by $n!^k(n-1)!(k-1)!$, by 
\tref{t:switch}.
\end{proof}

In particular, \tref{t:RSfromreps} can be used to find $\RS_{k,n}$
from a set of species representatives for sets of \MOLS{k}{n}, using
\texttt{nauty} to find $|\ptp(M)|$ for each representative $M$. We can then
employ \tref{t:switch} to find $\RL_{k,n}$, $\AL_{k,n}$ and
$\AS_{k,n}$.

\section{Pairs of MOLS of order 9}\label{s:pairs9}

In this section we explain the most difficult part of our
computations, namely finding the number of pairs of MOLS of order $9$
modulo each of the equivalences defined in \sref{s:equiv}.
Throughout this section, MOLS will mean an ordered pair (list) of
reduced MOLS, and all latin squares will have order 9.

Each pair of MOLS has exactly four aspects.  We use the notation
$P[i]$ to denote the aspect that results from deleting the $i^{\rm
  th}$ column of the orthogonal array corresponding to a pair $P$ of
MOLS.


Unsurprisingly, symmetry plays a crucial role in our counting. For
this reason, one task was to find sets $\Gamma_1$ and $\Gamma_2$ of
species representatives for the symmetric latin squares and symmetric
MOLS, respectively (we stress that symmetric is used throughout in
the sense defined in \sref{s:equiv}, not in the usual matrix
sense). The authors of \cite{MMM07} collated $\Gamma_1$, which
contains 2523159 latin squares. We will explain below how we found
$\Gamma_2$, and then analysed it to deduce data on the rigid MOLS from
the overall totals.

Let $\Lambda$ be a set of species representatives of reduced latin squares 
of order $9$.  
Let $\Omega$ be the set of all pairs $(A,B)$ of reduced MOLS for which
$A\in\Lambda$. Using the method discussed in \sref{s:algo}, we generated
$\Omega$ and found that 
\begin{equation}\label{e:sizeOmega}
|\Omega|=\sum_{A\in\Lambda}\theta(A)=390255632.
\end{equation}
We did not store all of $\Omega$ but kept statistics from the
generation as well as a list (which will be defined shortly) of
candidates for members of $\Gamma_2$. In any MOLS with a non-trivial
autotopism group both latin squares have a non-trivial autotopism
group and hence are paratopic to a member of $\Gamma_1$.  Such MOLS
are relatively easy to generate directly from $\Gamma_1$.  Hence,
while generating $\Omega$, we only needed to find all MOLS that have
an autoparatopism that is not an autotopism. Such MOLS necessarily
have two paratopic aspects. Rather than the relatively time-consuming
task of calculating the autoparatopism group of each set of MOLS in
$\Omega$ we computed two species invariants for each aspect. First we
calculated the number of intercalates, and if that did not
discriminate between the aspects, we counted the number of
transversals. If any two of the four aspects agreed on both statistics
then we stored the MOLS as a candidate for being in $\Gamma_2$. These
candidates, together with the 25382851 MOLS $(A,B)$ for which
$A\in\Gamma_1$, were subsequently screened to produce $\Gamma_2$. As it
turned out, $|\Gamma_2|=257442$. A by-product of our method for finding
$\Gamma_2$ is that we were also able to identify MOLS with two paratopic
aspects even if there was no symmetry that mapped one to the other.
Data on this issue will be presented in \Tref{tbl:numspeciesinv}.


\begin{table}[htbp]
\begin{center}
\begin{tabular}{|c|cccc|}
  \hline
$|\ptp(A)|$&\#Species & \#Pairs & \#Symmetric & $\chi$ \\
\hline
 1& 19268330382& 364872781 & 70240 & 364802541\\
 2& 2497877& 2620967& 654163& 983402\\
 3& 15618& 77434& 42211& 11741\\
 4& 6890& 923949& 166421& 189382\\
 5& 12& & & \\
 6& 2237& 1010064& 65304& 157460\\
 7& 5& 7& & 1\\
 8& 151& 149780& 47940& 12730\\
 9& 10& 677& 434& 27\\
 10& 21& & & \\
 12& 196& 1807096& 122512& 140382\\
 14& 1& & & \\
 16& 10& 25392& 8224& 1073\\
 18& 43& 93779& 12923& 4492\\
 20& 3& & & \\
 21& 4& & & \\
 24& 28& 555467& 74291& 20049\\
 30& 4& & & \\
 32& 1& 284& 124& 5\\
 36& 23& 685034& 79838& 16811\\
 48& 1& 197& 149& 1\\
 54& 2& 187657& 16693& 3166\\
 60& 1& & & \\
 72& 6& 541584& 105192& 6061\\
 96& 2& 14568& 10152& 46\\
 108& 4& 260888& 27392& 2162\\
 162& 1& 3124& 2314& 5\\
 168& 1& 84& 84& \\
 216& 2& 544264& 105136& 2033\\
 324& 1& 139968& 81972& 179\\
 432& 1& 4171& 3739& 1\\
 972& 1& 1241361& 225621& 1045\\
 2916& 1& 2049219& 375435& 574\\
 23328& 1& 12445836& 4071084& 359\\
\hline
\hline
Total& 19270853541 & 390255632 & 6369588 & 366355728 \\
\hline
\end{tabular}
\caption{\label{tbl:pairs9data}Data for counting pairs $(A,B)$ of MOLS(9)}
\end{center}
\end{table}  

We next consider how many times a given species of MOLS will appear
in $\Omega$.

\begin{lemma}\label{l:aspects}
The number of MOLS in $\Omega$ that are paratopic to a given pair $P$ is
\[
\frac{1}{|\ptp(P)|}\sum_{i=1}^4\big|\ptp(P[i])\big|.
\]
\end{lemma}

\begin{proof}
  Let $G$ denote the paratopism group $\sym_n\wr\sym_4$ and let
  $H=\sym_n\wr(\sym_3\times S_1)$ be the subgroup of $G$ that
  preserves the species of $P[4]$, the first square in the pair $P$. For
  $g\in G$, let $P^g$ denote the image of $P$ under the action of $g$
  and let $P^G$ denote the orbit of $P$ under the action of $G$.  The
  quantity we seek is $|\Omega\cap P^G|$.  From the action of $H$ we
  see that each choice of $A$ from a species of latin squares has the
  same number of choices for $B$ for which $(A,B)\in P^G$. Hence
$$\big|\Omega\cap P^G\big|=\big|\{(A,B)\in P^G:A\in\Lambda\}\big|
=\sum_{(A,B)\in P^G}\frac{1}{|(A,B)^H|}
=\sum_{(A,B)\in P^G}\frac{|\ptp(A)|}{|H|},$$
by the Orbit-Stabiliser Theorem. 
Now
$$\sum_{(A,B)\in P^G}|\ptp(A)|
=\sum_{g\in G}\frac{|\ptp(P^g[4])|}{|\ptp(P)|}
=\sum_{i=1}^4\frac{|H||\ptp(P^g[i])|}{|\ptp(P)|},$$
from which the result follows.
\end{proof}

\Tref{tbl:pairs9data} shows some of the data that was used to
calculate the number of pairs of MOLS of order 9. In it, MOLS are
classified according to $g=|\ptp(A)|$, the order of the autoparatopism
group of the first latin square in the pair.  The value of $g$ is listed in the
first column. The second column counts how many species of latin squares have
autoparatopism group of size $g$ (this data was first calculated in
\cite{MMM07}).  The third column records the number of MOLS,
in other words, $\sum\theta(A)$ over all $A\in\Lambda$ with
$|\ptp(A)|=g$. The value for $g=1$ was deduced from \eref{e:sizeOmega}
and the values for larger $g$.  The fourth column lists how many 
symmetric MOLS were counted in the third column. This
information was obtained by applying \lref{l:aspects} to $\Gamma_2$.
The last column of \Tref{tbl:pairs9data} is headed $\chi$.  It is
calculated by subtracting the fourth column from the third column,
then dividing by $g$ (the first column). By \lref{l:aspects}, the
total $\chi$, namely 366355728, is four times the number of rigid
MOLS, which must therefore be 91588932. Together with the
$|\Gamma_2|=257442$ species of symmetric MOLS, this shows that there 
are a total of 91846374 species of \MOLS{2}{9}.

\begin{table}[htbp]
\begin{minipage}[t]{0.48\linewidth}
\begin{tabular}{|cc|cc|c|}
  \hline
$|\ptp|$&$|\atp|$&all&non-max\\
\hline
1&1&91588932&3\\
2&1&72273&12\\
2&2&156009&18\\
3&1&1859&3\\
3&3&17346&40\\
4&1&25&1\\
4&2&4923&32\\
4&4&411&1\\
6&1&302&7\\
6&2&275&\\
6&3&1522&28\\
6&6&1074&90\\
8&1&2&2\\
8&2&111&10\\
8&4&123&3\\
8&8&1&\\
9&3&103&3\\
9&9&256&18\\
12&2&51&4\\
12&3&1&1\\
12&4&6&\\
12&6&228&32\\
12&12&4&\\
16&2&9&2\\
16&4&37&\\
16&8&5&2\\
18&3&75&6\\
18&6&43&6\\
18&9&101&18\\
18&18&50&24\\
24&1&1&1\\
24&6&10&1\\
24&12&6&\\
27&27&12&\\
32&4&4&1\\
32&8&4&3\\
36&6&27&6\\
\hline
\end{tabular}
\end{minipage}
\quad
\begin{minipage}[t]{0.48\linewidth}
\begin{tabular}{|cc|cc|c|}
  \hline
$|\ptp|$&$|\atp|$&all&non-max\\
\hline
36&12&2&\\
36&18&40&15\\
48&2&3&2\\
48&4&1&\\
48&6&7&\\
48&8&3&2\\
48&12&1&\\
54&9&19&2\\
54&27&15&1\\
54&54&2&1\\
64&8&4&1\\
72&12&1&\\
72&18&7&7\\
72&36&1&1\\
81&27&2&\\
96&12&2&\\
108&18&9&1\\
108&54&7&3\\
144&36&1&1\\
162&27&4&3\\
162&54&1&\\
162&162&1&\\
216&36&1&\\
216&54&3&2\\
288&48&1&1\\
324&54&1&\\
384&48&1&1\\
432&54&1&1\\
432&72&2&2\\
486&81&1&\\
3888&486&1&\\
576&72&1&1\\
972&162&4&4\\
5184&648&1&1\\
11664&486&1&1\\
93312&3888&1&1\\
\hline
\hline
&Total&91846374&433\\
\hline
\end{tabular}\\
\end{minipage}
\caption{\label{tbl:pairs9bysym}Species of \MOLS{2}{9} categorised by symmetry}
\end{table}  

\Tref{tbl:pairs9bysym} shows the 91846374 species of MOLS categorised
by the sizes of their autoparatopism group and autotopism group
($|\ptp|$ and $|\atp|$, respectively).  For each combination of these
group sizes, the table lists the total number of species with groups of
those sizes and also the number of species of non-maximal MOLS. There
are only 433 species of non-maximal MOLS. These were identified by
screening $\Omega$ as it was produced, with all non-maximal MOLS that
we encountered then being stored in a separate file for later
analysis, including construction of all larger sets of MOLS.  
We stress that the values of $|\ptp|$ and $|\atp|$ given in
\Tref{tbl:pairs9bysym} are for lists rather than sets of MOLS (cf. the
example at the end of \sref{s:equiv}).

\begin{table}[tb]
\begin{tabular}{c|ccc}
\hline
order&
\parbox[m]{38mm}{Proportion of species\tspacer\\ that have a mate}&
\parbox[m]{48mm}{Probability of a random\tspacer\\ latin square having a mate}&
\parbox[m]{36mm}{Expected number\tspacer\\ of mates}\\[1.5ex]
\hline
$3$&$1$
&$1$
&$1$\\
$4$&$\frac{1}{2}$
&$\frac{1}{4}$
&$\frac{1}{2}$\\
$5$&$\frac{1}{2}$\spacer
&$\frac{3}{28}\approx 0.107143$
&$\frac{9}{28}\approx 0.321429$\\
$6$&$0$
&$0$
&$0$\\
$7$&$\frac{6}{147}\approx 0.040816$\spacer
&$\frac{5891}{564736}\approx 0.010431$
&$\frac{1427}{70592}\approx 0.020215$\\
$8$&$\frac{2024}{283657}\approx 0.007135$\spacer
&$\frac{103065585}{22303391744}\approx 0.004621$
&$\frac{40888485}{2787923968}\approx 0.014666$\\
$9$&$\frac{348498052}{19270853541}\approx 0.018084$
&$\frac{23706924145915}{1311102676959232}\approx 0.018082$
&$\frac{24960190907155}{1311102676959232}\approx 0.019038$\spacer\\
\hline
\end{tabular}
\caption{\label{tbl:RLS}Data for random latin squares of order $3\le n\le9$}
\end{table}

\section{The number of orthogonal mates}\label{s:nummates}

In this section we provide data on the number of orthogonal mates for
latin squares of order up to and including 9.  Data on the number of
species of bachelor latin squares of order $n\le9$ was first published
in \cite{ew2011lswrt}.  Here we calculate the probability of a
uniformly random latin square having an orthogonal mate, and the
expected number of mates. This is a simple calculation where each
species is weighted by the number of latin squares in that species in
order to calculate statistics across the whole set of latin squares of
a given order. The results are given in \Tref{tbl:RLS}.  It is
noteworthy that in \cite{MMM07} it is estimated that around 60\% of
latin squares of order 10 have mates and the expected number of mates
in a random latin square of order 10 exceeds 1. The values for 
orders in the range $5\le n\le9$ are clearly a lot smaller than this.

The only latin squares of order less than 7 that have orthogonal mates
are isotopic to the cyclic group of order 3 (which has 1 mate), the
elementary abelian group of order 4 (2 mates) or the cyclic group of
order 5 (3 mates).  Hence for the remainder of this section we
concentrate on the range $7\le n\le9$.

\begin{table}[tb]
     \setlength{\tabcolsep}{2\tabcolsep}
\begin{center}
\begin{tabular}{c|rrr}
         \hline
$r$&\multicolumn{1}{c}{7}&\multicolumn{1}{c}{8}&\multicolumn{1}{c}{9}\\
\hline
0&1&1223&336634416\\
1&3& 329&11654552\\
2& & 175&123054\\
3&1&  90&38700\\
4& &  67&20131\\
5& &  49&10913\\
6& &  31&7672\\
7& &  17&4552\\
8& &  15&2141\\
9&1&  7&902\\
10& &  4&341\\
\hline
\end{tabular}
\qquad
\begin{tabular}{c|rrr}
         \hline
$r$&\multicolumn{1}{c}{7}&\multicolumn{1}{c}{8}&\multicolumn{1}{c}{9}\\
\hline
11& &  6&379\\
12& &  5&217\\
13& &  1&30\\
14& &  3&6\\
15& &  1& \\
16& &  1&31\\
17& &   &10\\
18& &   &2\\
20& &   &2\\
23& &   &1\\
\hline
Total &6&2024& 348498052\\
\end{tabular}
\caption{\label{tbl:log2theta}
Non-bachelor species of order $7\le n\le9$ grouped by 
$r=\lfloor\log_2(\theta)\rfloor$}
\end{center}
\end{table}

For the latin squares of order $7\le n\le 9$ that have $\theta>0$
mates, we provide a summary of the number of mates in
\Tref{tbl:log2theta}. Since $\theta$ takes many different values for
these squares and the distribution is distinctly skewed towards
smaller values, we have grouped the counts according to the value of
$r=\lfloor\log_2(\theta)\rfloor$. In other words, for each $r$ the
table reports the number of different species for which the number of
mates lies in the interval $[2^r,2^{r+1})$.

It is not surprising that the latin squares with the most orthogonal
mates tend to have nice algebraic structure. The two species of order
9 with the most mates contain the elementary abelian group ($12445836$
mates) and the cyclic group ($2049219$ mates). The species with the
third highest number of mates ($1241361$) contains the 3
non-associative conjugacy-closed loops of order 9 (see \cite{CCloop}
for a definition of these loops).  Below that, the sequence of the number
of mates continues 403056, 277788, 253276, 242832, 237786, 226822,
207297,$\dots$.  There are 74 species with at least 10000 mates and
every one of them has a non-trivial autotopism group and
at least 4 subsquares of order 3. The species with the largest number
of mates and a trivial autotopism group (in fact, it is rigid) has
8226 mates, 6 subsquares of order 3 and 371 transversals.  A
representative of this species is
\[
\left[\setlength{\arraycolsep}{0.25\arraycolsep}
\begin{array}{ccccccccc}
\shdA0&1&2&3&4&\shdA5&\shdA6&7&8\\ 
1&\shdB2&0&4&5&\shdB3&\shdB7&8&6\\ 
2&0&\shdC1&5&3&\shdC4&\shdC8&6&7\\ 
\shdA5&8&3&2&7&\shdA6&\shdA0&1&4\\ 
\shdA6&4&7&8&1&\shdA0&\shdA5&3&2\\ 
8&\shdB7&6&1&0&\shdB2&\shdB3&4&5\\ 
4&\shdB3&5&6&8&\shdB7&\shdB2&0&1\\ 
7&6&\shdC8&0&2&\shdC1&\shdC4&5&3\\ 
3&5&\shdC4&7&6&\shdC8&\shdC1&2&0\\
\end{array}
\right],
\]
where shading indicates the subsquares of order 3 other than those
formed by the first 3 rows. Among the species with no subsquares of
order 3, the one with the most mates (4171) is the planar species $d$, which
has 72 subsquares of order 2, the maximum possible number 
\cite{McKayMcLeodWanless06}.

For order 8 the species with the three highest numbers of mates
contain the elementary abelian group (70272 mates), dihedral group
(33408 mates) and quaternion group (32256 mates), respectively. In
fourth place, with 23232 mates, is a species containing a loop that is
nearly a group in the sense that it has a large nucleus (isomorphic to
the Klein $4$-group).
The species of the group $\Z_4\times\Z_2$ is in fifth place (23040
mates). The top five places are occupied by the only latin squares
with $384$ transversals, which is the most that any latin square of order
$8$ has. The next highest number of mates is 12048.  The cyclic group,
of course, has no transversals and hence no mates.


For order 7 the species with the highest numbers of mates contain the
cyclic group (133 transversals, 63 mates), the Steiner quasigroup (63
transversals, 8 mates) and the pan-hamiltonian latin square that
is not atomic (25 transversals, 3 mates) (see $A_7$ in \cite{Wan99}).

\section{Number of sets of MOLS and maxMOLS}\label{s:MaxMOLS}

\begin{table}[htb]
\begin{center}
\begin{tabular}{cc|cccccc}
  \hline
  $n$ & $k$ & Equality & Isotopism  & Trisotopism  & Paratopism & PP & DPP\\
  \hline
  2&1&1&1&1&1&1&1\\[1ex]
  3&2&1&1&1&1&1&1\\[1ex]
  4&1&3&1&1&1&\\
  4&3&1&1&1&1&1&1\\[1ex]
  5&1&50&1&1&1&\\
  5&4&6&1&1&1&1&1\\[1ex]
  6&1&9408&22&17&12&\\[1ex]
  7&1&16765350&549&314&141&\\
  7&2&341880&17&11&5\\
  7&6&120&1&1&1&1&1\\[1ex]
  8&1&532807827816&1665394&836595&281633&\\
  8&2&7832534400&23005&11704&2127\\
  8&3&14923440&221&147&38\\
  8&7&240&1&1&1&1&1\\[1ex]
  9&1&370769976810235296&113527931950&56764991345&18922355489\\
  9&2&7188529229970480&1101731294&550905816&91845941\\
  9&3&7648799760&2943&1578&232\\
  9&4&665884800&371&203&22\\
  9&5&222499200&318&200&36\\
  9&8&7728840&19&15&7&4&3\\
\hline
\end{tabular}
\caption{\label{tbl:nummaxMOLS}Number of reduced sets of \maxMOLS{k}{n}} 
\end{center}
\end{table}  

For $1\le k<n\le9$, \Tref{tbl:nummaxMOLS} gives the number of 
reduced sets of \maxMOLS{k}{n} modulo each of the different notions of
equivalence defined in \sref{s:equiv}.  The column headed ``Equality''
gives the total number of reduced sets of \maxMOLS{k}{n}, in other
words, $\RS_{k,n}$. These numbers were calculated from a list of
species representatives using \tref{t:RSfromreps}.  The
number of (not necessarily reduced) sets of \maxMOLS{k}{n} can be
found from $\RS_{k,n}$ using \tref{t:switch}, as can the number of
lists of \maxMOLS{k}{n} (reduced or otherwise).

\begin{table}[htbp]
\begin{center}
\begin{tabular}{cc|cccccc}
  \hline
  $n$ & $k$ & Equality & Isotopism  & Trisotopism  & Paratopism & PP & DPP\\
  \hline
  2&1&1&1&1&1&1&1\\[1ex]
  3&1&1&1&1&1&&\\
  3&2&1&1&1&1&1&1\\[1ex]
  4&1&4&2&2&2&\\
  4&2&2&1&1&1&&\\
  4&3&1&1&1&1&1&1\\[1ex]
  5&1&56&2&2&2&\\
  5&2&18&2&2&1&&\\
  5&3&18&1&1&1&&\\
  5&4&6&1&1&1&1&1\\[1ex]
  6&1&9408&22&17&12&\\[1ex]
  7&1&16942080&564&324&147&\\
  7&2&342480&20&14&7\\
  7&3&1200&4&3&1\\
  7&4&1200&3&3&1\\
  7&5&600&1&1&1\\
  7&6&120&1&1&1&1&1\\[1ex]
  8&1&535281401856&1676267&842227&283657&\\
  8&2&7850589120&23362&11887&2165\\
  8&3&14927040&224&149&39\\
  8&4&4800&3&2&1\\
  8&5&3600&1&1&1\\
  8&6&1440&1&1&1\\
  8&7&240&1&1&1&1&1\\[1ex]
  9&1&377597570964258816&115618721533&57810418543&19270853541&\\
  9&2&7188534981260640&1101734942&550907773&91846374&\\
  9&3&9338177520&4428&2408&371\\
  9&4&1526884800&1096&642&96\\
  9&5&493008600&454&293&56\\
  9&6&162305640&82&62&15\\
  9&7&54101880&38&29&11\\
  9&8&7728840&19&15&7&4&3\\
\hline
\end{tabular}
\caption{\label{tbl:numMOLS}Number of reduced sets of \MOLS{k}{n}} 
\end{center}
\end{table}

\begin{table}[htb] 
\begin{center}
\begin{tabular}{cc|cccccc}
  \hline
  $n$ & $k$ & Equality & Isotopism  & Trisotopism  & Paratopism & PP & DPP\\
  \hline
  2&1&1&1&1&1&1&1\\[1ex]
  3&2&1&1&1&1&1&1\\[1ex]
  4&1&3&1&1&1&\\
  4&3&2&1&1&1&1&1\\[1ex]
  5&1&50&1&1&1&\\
  5&4&36&6&3&1&1&1\\[1ex]
  6&1&9408&22&17&12&\\[1ex]
  7&1&16765350&549&314&141&\\
  7&2&341880&29&17&5\\
  7&6&14400&120&60&1&1&1\\[1ex]
  8&1&532807827816&1665394&836595&281633&\\
  8&2&7832534400&45222&23005&2127\\
  8&3&29846880&1217&616&38\\
  8&7&172800&240&120&1&1&1\\[1ex]
  9&1&370769976810235296&113527931950&56764991345&18922355489\\
  9&2&7188529229970480&2203304036&1101731294&91845941\\
  9&3&15297599520&15963&8228&232\\
  9&4&3995308800&8150&4111&22\\
  9&5&5339980800&18060&9030&36\\
  9&8&38953353600&56700&28350&7&4&3\\
\hline
\end{tabular}
\caption{\label{tbl:nummaxMOLSlist}Number of reduced lists of \maxMOLS{k}{n}} 
\end{center}
\end{table}  

\begin{table}[htbp]
\begin{center}
\begin{tabular}{cc|cccccc}
  \hline
  $n$ & $k$ & Equality & Isotopism  & Trisotopism  & Paratopism & PP & DPP\\
  \hline
  2&1&1&1&1&1&1&1\\[1ex]
  3&1&1&1&1&1&&\\
  3&2&1&1&1&1&1&1\\[1ex]
  4&1&4&2&2&2&\\
  4&2&2&1&1&1&&\\
  4&3&2&1&1&1&1&1\\[1ex]
  5&1&56&2&2&2&\\
  5&2&18&3&2&1&&\\
  5&3&36&6&3&1&&\\
  5&4&36&6&3&1&1&1\\[1ex]
  6&1&9408&22&17&12&\\[1ex]
  7&1&16942080&564&324&147&\\
  7&2&342480&34&20&7\\
  7&3&2400&20&10&1\\
  7&4&7200&60&30&1\\
  7&5&14400&120&60&1\\
  7&6&14400&120&60&1&1&1\\[1ex]
  8&1&535281401856&1676267&842227&283657&\\
  8&2&7850589120&45927&23362&2165\\
  8&3&29854080&1227&621&39\\
  8&4&28800&40&20&1\\
  8&5&86400&120&60&1\\
  8&6&172800&240&120&1\\
  8&7&172800&240&120&1&1&1\\[1ex]
  9&1&377597570964258816&115618721533&57810418543&19270853541&\\
  9&2&7188534981260640&2203310919&1101734942&91846374&\\
  9&3&18676355040&23677&12264&371\\
  9&4&9161308800&21705&10944&96\\
  9&5&11832206400&27510&13800&56\\
  9&6&19476676800&28350&14220&15\\
  9&7&38953353600&56700&28350&11\\
  9&8&38953353600&56700&28350&7&4&3\\
\hline
\end{tabular}
\caption{\label{tbl:numMOLSlist}Number of reduced lists of \MOLS{k}{n}} 
\end{center}
\end{table}

\Tref{tbl:numMOLS} gives, for each $1\le k<n\le9$, the number of
non-equivalent reduced sets of \MOLS{k}{n} under each of the different
notions of equivalence defined in \sref{s:equiv}. We stress that the
difference between \Tref{tbl:nummaxMOLS} and \Tref{tbl:numMOLS} is
that the former counts only {\em maximal} sets, while the latter also
includes sets that are not maximal.  Following those tables, we give
\Tref{tbl:nummaxMOLSlist} and \Tref{tbl:numMOLSlist} which provide 
the same information, except for lists of MOLS rather than sets of MOLS.
In Tables \ref{tbl:nummaxMOLS} to \ref{tbl:numMOLSlist}, the stipulation
that the MOLS should be reduced only affects the counts in the column
headed ``Equality''. Every isotopism class contains reduced MOLS so counting
reduced MOLS up to isotopism is the same as counting isotopism classes.
Similar statements apply to trisotopism classes and species.

\begin{table}[htb]
\begin{center}
\begin{tabular}{c|cc|c}
  \hline
  \#Common & \multicolumn{2}{c|}{\#Disjoint}&\\
  transversals		 &0  	&1  &Total\\
  \hline
  0&1	 	& 	&1\\
  1& 	 	&1	&1\\
  2& 	 	&1 	&1\\
  4& 	 	&2 	&2\\
  \hline
  Total&1		&4	&5\\
\end{tabular}
\caption{\label{tbl:commonTransv2maxMOLS7}\maxMOLS{2}{7} according to 
their common transversals} 
\end{center}
\end{table}

\begin{table}[htb]
\begin{center}
\begin{tabular}{c|cccc|c}
  \hline
  \#Common & \multicolumn{4}{c|}{\#Disjoint}&\\
  transversals		 &0  &1 	&2 	&4 &Total\\
  \hline
  0&1980& 	&  	&		&1980\\
  1& 	 &23	&  	&		&23\\
  2& 	 &10 	&60 &		&70\\
  3& 	 &1		&  	&		&1\\
  4& 	 &		&16 &26	&42\\
  8& 	 &		&1  &7	&8\\
  12& 	 &		&1  &1	&2\\
  19& 	 &		&1  &		&1\\
  \hline
  Total&1980&34	&79	&34	&2127\\
\end{tabular}
\caption{\label{tbl:commonTransv2maxMOLS8}\maxMOLS{2}{8} according to 
their common transversals}
\end{center}
\end{table}

\begin{table}[htb]
\begin{center}
\begin{tabular}{c|cccc|c}
  \hline
  \#Common & \multicolumn{4}{c|}{\#Disjoint}&\\
  transversals		 &0  &1 	&3 	&5 &Total\\
  \hline
  0&188& 		&  	&		&188\\
  1& 	 &8		&  	&		&8\\
  2& 	 &5 	&  	&		&5\\
  3& 	 &7		&5  &		&12\\
  4& 	 &6		&  	&		&6\\
  5& 	 &6		&  	&1	&7\\
  6& 	 &		&3  &		&3\\
  10& 	 &1		&  	&		&1\\
  11& 	 &		&1 	&		&1\\
  12& 	 &		&1  &		&1\\
  \hline
  Total&188&33	&10	&1	&232\\
\end{tabular}
     \caption{\label{tbl:commonTransv3maxMOLS9}\maxMOLS{3}{9} according to their common transversals}
\end{center}
\end{table}
    
\begin{table}[htb]
\begin{center}
\begin{tabular}{c|cc|c}
  \hline
  \#Common & \multicolumn{2}{c|}{\#Disjoint}&\\
  transversals		 &0  	&3 	&Total\\
  \hline
  0&21	& 	&21\\
  6& 	 	&1	&1\\
  \hline
  Total&21	&1	&22\\
\end{tabular}
\caption{\label{tbl:commonTransv4maxMOLS9}\maxMOLS{4}{9} according to their 
common transversals}
\end{center}
\end{table}

\begin{table}[htb]
\begin{center}
\begin{tabular}{c|cccccccc|c}
  \hline
  \#Common & \multicolumn{8}{c|}{\#Disjoint}&\\
  transversals		 &0&1&2	&3 &4&5&6&7	&Total\\
\hline
0& 183793 &&&&&&& & 183793 \\
1&& 14079 &&&&&& & 14079 \\
2&& 32580 & 1244 &&&&& & 33824 \\
3--4&&8051&3672&1605&9&&&&13337\\
5--8&&2397&3128&1756&74&47&22&1&7425\\
9--16&&483&1023&1328&253&69&67&6&3229\\
17--32&&140&210&457&157&79&75&2&1120\\
33--60&&13&20&97&13&2&9&2&156\\
66--120&&20&3&21&2&1&&1&48\\
216&&&&1&&&&&1\\
  \hline
  Total& 183793 & 57763 & 9300 & 5265 & 508 & 198 & 173 & 12 & 257012 \\
\end{tabular}
\caption{\label{tbl:commonTransv2maxMOLS9}Symmetric \maxMOLS{2}{9} according to their common transversals}    
\end{center}
\end{table}

For the remainder of this discussion we count all MOLS by species. 
Hedayat, Parker and Federer
\cite{HedayatParkerFederer1970} showed how sets of disjoint common
transversals of a set of MOLS can be used to design successive
experiments.  In \Tref{tbl:commonTransv2maxMOLS7}, the $5$ sets of
\maxMOLS{2}{7} are classified according to their number of common
transversals and maximum number of disjoint common transversals.  We
present similar tables for the 2127 sets of \maxMOLS{2}{8}
(\Tref{tbl:commonTransv2maxMOLS8}), the 232 sets of \maxMOLS{3}{9}
(\Tref{tbl:commonTransv3maxMOLS9}) and the 22 sets of \maxMOLS{4}{9}
(\Tref{tbl:commonTransv4maxMOLS9}).  We do not provide tables for the
38 sets of \maxMOLS{3}{8} or the 36 sets of \maxMOLS{5}{9}, each of
which has no common transversal.  We also do not provide a table for
the 91845941 sets of \maxMOLS{2}{9}, since we did not collect data on
their common transversals. However, in
\Tref{tbl:commonTransv2maxMOLS9} we do summarise the {\em symmetric}
\maxMOLS{2}{9} according to their common transversals.

Our next aim is to examine how prevalent planar species of latin
squares are in MOLS. We say that a latin square $L$ is {\em involved}
in MOLS $M$ if at least one aspect of $M$ is paratopic to $L$. We say
that a set of MOLS has type P (respectively N) if every latin square
in the set of MOLS is planar (respectively, non-planar).  A set of
MOLS is of type M (for mixed) if it is neither of type P or N.  In
\Tref{T:planartype} we classify the species of \maxMOLS{k}{9}
according to which types of MOLS they contain. Types of MOLS that are
not listed are assumed to be not present. So, for example, the column
headed ``PM'' counts species of MOLS that contain at least one set of
MOLS of type P, at least one set of MOLS of type M, and no sets of
MOLS of type N.  It is worth remarking that there are no columns headed
``M'' or ``PN'' because no \maxMOLS{k}{9} fell in those
categories. There seems to be no obvious reason why ``M'' is impossible,
but we now describe an obstacle that prevents ``PN''
occurring. Suppose that we have an $n^2\times k$ orthogonal array $O$.
Let $O_{ij}$ be the set of MOLS obtained by taking column $i$ of $O$
to index the rows of our MOLS, and column $j$ of $O$ to index the
columns of our MOLS.  Suppose that $O_{12}$ is of type P and $O_{ab}$
is of type N, for some $1\le a<b\le k$. Then $O_{1b}$ is of type M
since it contains one latin square paratopic to an element of $O_{12}$
and another latin square paratopic to an element of $O_{ab}$.

\begin{table}[htb]
\begin{center}
\begin{tabular}{c|ccccc|c}
  \hline
  $k$& P&N&PM	&NM &PNM&Total\\
  \hline
  1&&18922355489&&&&18922355489 \\
  2&3&91835638&6&10224&70&91845941 \\
  3&1&39&3&186&3&232 \\
  4&&&3&4&15&22 \\
  5&7&&19&6&4&36 \\
  8&7&&&&&7 \\
\hline
\end{tabular}
\caption{\label{T:planartype}Species of 
\maxMOLS{k}{9} classified by planarity type}
\end{center}
\end{table}

Clearly, each of the 7 species of \maxMOLS{8}{9} involve only planar
latin squares.  Planar latin squares are also
involved in many of the \maxMOLS{k}{9} for $k\in\{3,4,5\}$.  In
particular, we can see from \Tref{T:planartype} that all 36 species of
\maxMOLS{5}{9} involve at least one planar latin square and seven of
them involve only planar latin squares.  All 22 species of
\maxMOLS{4}{9} involve at least one planar latin square and at least
one non-planar latin square.  
There are three species of \maxMOLS{3}{9} for which there is only one
species of latin square involved; in one case the sole species is the
planar species $e$, in the other two cases the species is not
planar. There is one species of \maxMOLS{5}{9} that involves only one
species of latin square (namely, the planar species $a$, the
elementary abelian group). All other \maxMOLS{5}{9} involve at least
two distinct planar species and between 3 and 9 (inclusive) species of
latin squares in total.

We next consider the possibility that a latin square $L$ may be in a
set of $\theta(L)+1$ MOLS. In other words, the set of all orthogonal
mates for $L$ itself forms a set of MOLS. This is automatically true
if $\theta(L)=1$ but we would expect it to be rare for larger values
of $\theta$.  For order $9$ we have the following data:
\begin{itemize}\addtolength{\itemsep}{-0.25\baselineskip}
\item There are exactly 11222874 species of order 9 that possess
  exactly two mates.  Of these, 27 species appear in a set of
  \maxMOLS{3}{9}.
\item Of the 431678 species with $\theta = 3$, none occur in a set of
  \maxMOLS{4}{9}.
\item Of the 74741 species with $\theta = 4$, precisely one species is
  in a set of \maxMOLS{5}{9}. A representative of that species is
\[
\left[\setlength{\arraycolsep}{0.25\arraycolsep}
\begin{array}{ccccccccc}
0&1&2&3&4&5&6&7&8\\
1&2&0&7&5&4&8&3&6\\
2&0&1&6&8&7&3&4&5\\
3&8&6&4&0&1&2&5&7\\
4&5&7&0&3&8&1&6&2\\
5&4&8&1&7&6&0&2&3\\
6&7&3&2&1&0&5&8&4\\
7&6&5&8&2&3&4&0&1\\
8&3&4&5&6&2&7&1&0\\
\end{array}
\right].
\]
It has 242 transversals, 3 subsquares of order 3 (all including the entry
in the top left corner) and an autoparatopism group of order 4.
\end{itemize}

\begin{table}[htb]
\begin{center}
\begin{tabular}{cc||c|c|c|cc|ccc|ccccc}
  \hline
 & $n$& \multicolumn{1}{c|}{3} & \multicolumn{1}{c|}{4}&\multicolumn{1}{c|}{5}& 
\multicolumn{2}{c|}{7} & \multicolumn{3}{c|}{8} & \multicolumn{5}{c}{9} \\
 \#LS& $k$&2&3&4 & 2& 6&  2& 3 & 7 &  2& 3 & 4 & 5 & 8\\
  \hline
  1&&1&1&1&2&1&4	&1	&1&116&	3	&	&1	&2\\
  2&&&&&2& &82	&6	&       &5953&10	&1	&	&2\\
  3&&&&&1& &512&	13     &&100971&22	&1	&12	&2\\
  4&&&&& & &1529	&16    &&91738901&44	&5	&12	&\\
  5&&&&& & &	&2	&    & &30	&5	&8	&1\\
  6&&&&& & &        &       && &62	&5	&1	&\\ 
  7&&&&& & &        &       && &38	&1	&1	&\\ 
  8&&&&& & &        &       && &18	&2	&	&\\ 
  9&&&&& & &        &       && &1	&2	&1	&\\ 
 10&&&&& & &        &       && &4	&	&	&\\
  \hline
  Total &&1&1&1&5	&1&2127	&38 &1 &91845941&232	&22	&36	&7\\
\end{tabular}
\caption{\label{tbl:numspeciesinv}Number of species of LS involved 
in the species of \maxMOLS{k}{n}} 
\end{center}
\end{table}

A set of \maxMOLS{k}{n} has $\binom{k+2}{3}$ different
aspects that may potentially belong to different species.
In \Tref{tbl:numspeciesinv} we show, for $2\le k<n\le 9$, how many
different species of latin squares are involved in each species of
\maxMOLS{k}{n}. The number, say $s$, of species of latin squares is listed in
the first column of \Tref{tbl:numspeciesinv}, while other columns 
list the number of species of \maxMOLS{k}{n} which involve exactly $s$ 
different species of latin squares. It seems from the table that it is
fairly common for pairs of MOLS to have aspects in $4$ different species.
However, for $k>2$ the theoretical bound of $\binom{k+2}{3}$ different
species is rarely attained among the cases covered by \Tref{tbl:numspeciesinv}.

\begin{table}[htb]
\begin{center}
\begin{tabular}{ccccccccc}
  \hline   
Species 
&	Transversals	&	$\alpha$	&	$\theta$ &$k=2$	& $k=3$ & $k=4$ & $k=5$ & $k=8$\\
\hline  
$a$	 &	2241	&	5	&	12445836	&935&	69	&	16&	34	&	5	  \\
$b$	 &	417	&	4	&	11448	        &265&	1		&		&			&	2	  \\
$c$	 &	489	&	4	&	197		&9& 3		&		&	1		&	1	  \\
$d$	 &	801	&	4	&	4171		&20&	9		&	1	&	7		&	2	  \\
$e$	 &	553	&	4	&	3120		&87&	9		&		&			&	1	  \\
$f$	 &	405	&	3	&	8928		&200&	69	&	7	&	22	&	1	  \\
$g$	 &	1620	&	4	&	1241361		&1816&	94	&	18&	30	&	1	  \\
$h$	 &	861	&	4	&	242832		&4248&	9		&	2	&	1		&	1	  \\
$i$	 &	351	&	4	&	2886		&424&	1		&	1	&			&	1	  \\
$j$	 &	369	&	4	&	59		&12&			&	1	&	1		&	1	  \\
$k$	 &	855	&	4	&	403056		&2335& 2		&	1	&	2		&	1	  \\
\hline
non-planar: &		&		&			&	& 		&		&			&		  \\
$\Z_9$ &2025	&5		&2049219			&932&6 		&5		&6			&		  \\	
$\spc$ &819		&4		&	141208		&863&77 		&16		&20			&		  \\
\hline
\end{tabular}
\caption{\label{tbl:planarPlusOAequiv}Statistics on selected species of order $9$ including the number of species of \maxMOLS{k}{9} in which they occur} 
\end{center}
\end{table}

In \Tref{tbl:planarPlusOAequiv} we record statistics on a selection of
species of order $9$.  The species $\{a,b, \ldots, k\}$ are the planar
species according to their alphabetic label given in
\cite{OP97}.  The other two species referred to in
\Tref{tbl:planarPlusOAequiv} are the species of the Cayley table of
$\Z_9$ and a species we call $\spc$, which occurs with high frequency
among \maxMOLS{k}{9} for $k\in\{3,4,5\}$.
Each square in species $\spc$ has 18 subsquares of order 3 (and none
of order 2). A representative of $\spc$ is
\[\left[\setlength{\arraycolsep}{0.25\arraycolsep}
\begin{array}{ccccccccc}
0&1&2&3&4&5&6&7&8\\
1&2&0&4&5&3&7&8&6\\
2&0&1&5&3&4&8&6&7\\
3&5&4&6&7&8&0&2&1\\
4&3&5&7&8&6&1&0&2\\
5&4&3&8&6&7&2&1&0\\
6&8&7&0&2&1&3&4&5\\
7&6&8&1&0&2&4&5&3\\
8&7&6&2&1&0&5&3&4
\end{array}
\right]
\]
\def\li{\star}%
When considered as a loop, it has the antiautomorphic inverse
property.  This means that it satisfies the law $(xy)^\li=y^\li x^\li$, 
for all $x$ and $y$, where $^\li$ denotes the left inverse.
In \Tref{tbl:planarPlusOAequiv} we give the number of transversals for
each species. Next we give the value of $\alpha$, which is
the smallest number of transversals in a maximal set of disjoint 
transversals (see \cite{ew2011lswrt}). After that, we give $\theta$,
the number of orthogonal mates. The remaining columns count how many 
species of \maxMOLS{k}{9} include the given species of latin square.

An interesting feature of \Tref{tbl:planarPlusOAequiv} is that planar
species $a$ has an order of magnitude more mates than any other latin
square, but is a long way from being involved in the most species
of \maxMOLS{2}{9}. In fact that honour does not go to any of the species
covered in the table, but rather to the species represented by
\[\left[\setlength{\arraycolsep}{0.25\arraycolsep}
\begin{array}{ccccccccc}
0&1&2&3&4&5&6&7&8\\
1&2&3&4&5&6&7&8&0\\
2&0&1&5&3&7&8&6&4\\
3&4&8&6&7&2&0&1&5\\
4&8&0&7&2&3&1&5&6\\
5&6&4&8&0&1&2&3&7\\
6&7&5&0&1&8&3&4&2\\
7&5&6&1&8&0&4&2&3\\
8&3&7&2&6&4&5&0&1\\
\end{array}
\right].
\]
This square has 755 transversals and an autoparatopism group of order 2.
It has 121330 mates and is in 58296 different species of \MOLS{2}{9}, all of them maximal.

\section{Crosschecking}\label{s:safety}

Any computation runs the risk of errors, with the risk increasing with
the length and complexity of the computation. The following
precautions and crosschecks have been implemented to try to minimise the
risk of errors affecting our results.

\begin{itemize}\addtolength{\itemsep}{-0.5\baselineskip}

\item Data in all of the tables was computed at least twice.
  There was some common code used, most notably the generator of
  species representatives from \cite{MMM07} and the code for
  screening MOLS for isomorphism. Both of these programs have been
  previously used for multiple tasks, reducing the likelihood that
  bugs would have been undetected. With the caveat that this code was
  common, the main computations were performed independently. For example,
  both authors found their own versions of the set $\Gamma_2$, which were
  then compared to check that each set contained the same species of MOLS.

\item After we generated our catalogues, Brendan McKay kindly gave us
  code he had written for canonically labelling MOLS and
  calculating their autoparatopism and autotopism groups. With this
  code we were able to verify that MOLS in our catalogues of
  representatives really were from distinct species or isotopism
  classes, as appropriate. We also checked that our code agreed with
  his on all group sizes, including those in \Tref{tbl:pairs9bysym}.

\item We found \maxMOLS{k}{n} exist exactly when the prior literature
(see \sref{s:intro}) said they should.

\item The number of species of \MOLS{2}{n} had previously been
  computed by Brendan McKay \cite{BDMLS} for $n\le8$. His results agree
  with ours in \Tref{tbl:numMOLS} and \Tref{tbl:numMOLSlist}. 

\item Norton \cite{Nor39} manually enumerated lists of MOLS of order
  7. His enumeration of species of latin squares of order 7 was
  incomplete, but the single species that he missed contains bachelor
  latin squares, so this did not affect his results on MOLS. His
  values for the number of species, isotopism classes and reduced
  latin squares agree with ours in \Tref{tbl:numMOLSlist} for $2\le
  k<7=n$. He also calculated that $\AL_{2,7}=6263668776960000$, which
  agrees with the value given by \tref{t:switch} from our value of
  $\RL_{2,7}$.

\item A number of our computations confirm results obtained by Owens
  and Preece for sets of \MOLS{8}{9}. It was reported in
  \cite{OP95} that there are 19 isotopism classes (in 7
  species) of sets of \MOLS{8}{9}. This agrees with our results in the
  final line of \Tref{tbl:nummaxMOLS} and \Tref{tbl:numMOLS}.  Also,
  the last column of \Tref{tbl:numspeciesinv} tallies with
  \cite[Table~4]{OP97}.

\item The total of the $\chi$ column in \Tref{tbl:pairs9data}
  evaluated to a multiple of four, as it should. If that total had
  been corrupted by one or more errors, it is quite likely that
  the result would not be divisible by $4$.

\item For each $n$, the number of isotopism classes of sets of \MOLS{2}{n}
equals the number of trisotopism classes of lists of \MOLS{2}{n}. A similar
equality holds if attention is restricted to \maxMOLS{2}{n}. Thus there
are several cases of equinumerous objects being counted in
Tables~\ref{tbl:nummaxMOLS} to \ref{tbl:numMOLSlist}.
The reason can be seen by considering the corresponding orthogonal arrays
and which operations result in equivalence. For isotopism of sets we allow the
last two columns of the orthogonal array to be exchanged, whereas for 
trisotopism of lists we allow the first two columns to be exchanged.
In other respects the two cases are identical. Hence, reversing the order
of the columns of the orthogonal arrays provides a bijection between
the objects that we claimed are equinumerous.

\item Hicks et al.~\cite{HMSV} show that there are exactly $(p^d-2)!/d$
reduced sets of \MOLS{(p^d-1)}{p^d} that define the Desarguesian 
projective plane of order $p^d$. Our computations agreed for $p^d\le9$.

\item The method outlined in \sref{s:pairs9} for counting the pairs of
  MOLS of order 9 was replicated for order 8 and agreed with the results
  of our more direct computations. Smaller orders do not provide
  useful test cases, since there are no rigid MOLS of order $n\le7$.


\end{itemize}

Data from our enumerations is available online at \cite{WWWW}, 
including species representatives for the MOLS that we generated.

\section{Order 10}\label{s:ten}
 
For orders 10 and higher there are simply too many latin squares to
attempt the sorts of comprehensive enumerations of the sort undertaken
in the previous sections. However, given the tremendous interest in
the existence or otherwise of a triple of MOLS of order 10 (see
\cite{MMM07} and the references therein), we did use our programs to
investigate the latin squares with autoparatopism groups of order at
least 3. A catalogue of these squares was produced by the authors of
\cite{MMM07}.  It was already established in \cite{MMM07} that none of
these squares is in any triple of MOLS. However, some of them come
much closer than any previously known examples, as we
discovered. Consider the following three squares
\[\setlength{\arraycolsep}{0.25\arraycolsep}
A=\left[\begin{array}{cccccccccc}
0&8&9&7&5&6&4&2&3&1\\
9&1&4&6&2&7&3&8&0&5\\
7&4&2&5&1&3&8&6&9&0\\
8&6&5&3&9&2&1&0&4&7\\
6&2&1&8&4&0&9&5&7&3\\
4&9&3&2&7&5&0&1&6&8\\
5&3&7&1&0&8&6&9&2&4\\
3&5&0&9&8&4&2&7&1&6\\
1&7&6&0&3&9&5&4&8&2\\
2&0&8&4&6&1&7&3&5&9\\
\end{array}
\right],\quad
B=\left[\begin{array}{cccccccccc}
0&7&8&9&1&2&3&4&5&6\\
9&0&6&1&8&3&2&5&4&7\\
7&2&0&4&3&9&1&8&6&5\\
8&5&3&0&2&1&7&6&9&4\\
6&9&5&3&0&7&4&2&1&8\\
4&1&7&6&5&0&8&9&3&2\\
5&4&2&8&9&6&0&3&7&1\\
3&6&1&7&4&8&5&0&2&9\\
1&8&4&2&6&5&9&7&0&3\\
2&3&9&5&7&4&6&1&8&0\\
\end{array}
\right],\quad
C=\left[\begin{array}{cccccccccc}
0&7&8&9&1&2&3&4&5&6\\
6&4&2&8&9&5&1&3&7&0\\
4&9&5&3&2&7&6&0&1&8\\
5&1&7&6&4&3&8&9&0&2\\
3&2&9&0&7&1&5&6&8&4\\
1&0&3&7&6&8&2&5&4&9\\
2&8&0&1&3&4&9&7&6&5\\
9&5&4&2&8&6&0&1&3&7\\
7&3&6&5&0&9&4&8&2&1\\
8&6&1&4&5&0&7&2&9&3\\
\end{array}
\right].
\]
Square $A$ is orthogonal to both $B$ and $C$. When $B$ and $C$ are overlayed,
91 different pairs are produced out of a possible 100. Moreover, the only
duplicated pairs involve symbols 7,8,9 in $C$. We conclude that $A$ and $B$
have 7 common transversals. The previously best published result showed a pair of MOLS of order 10 with 4 common transversals \cite{BrownHedayatParker1993}.

Note that $A$ is semisymmetric and $A$, $B$ and $C$ all have the
automorphism $(0)(123)(456)(789)$.

\bigskip

\noindent
\textbf{Acknowledgments.} 
This research was supported in part by an Australian Mathematical
Society Lift-Off Fellowship and by ARC grant FT110100065.
Computations were performed mainly on the Monash Sun Grid and Monash
Green SPONGE computing facilities.  The authors are deeply indebted to
Darcy Best for carefully checking a number of the tables.  The authors
are also grateful to Brendan McKay for making available his code for
canonically labelling orthogonal arrays, and to Patric Osterg{\aa}rd
and Petteri Kaski for helpful discussions.  We would also like to
thank Wendy Myrvold for independently confirming the results of
Section~\ref{s:ten} and for providing the program to generate species
representatives from~\cite{MMM07}.

\bibliographystyle{plain}
  \let\oldthebibliography=\thebibliography
  \let\endoldthebibliography=\endthebibliography
  \renewenvironment{thebibliography}[1]{%
    \begin{oldthebibliography}{#1}%
      \setlength{\parskip}{0.2ex plus 0.1ex minus 0.1ex}%
      \setlength{\itemsep}{0.2ex plus 0.1ex minus 0.1ex}%
  }%
  {%
    \end{oldthebibliography}%
  }

\end{document}